\documentclass[12pt]{article}
\usepackage[T1]{fontenc}
\usepackage{amssymb} 
\usepackage{amssymb}
\usepackage{amsmath, amscd}
\usepackage{color}
\usepackage{latexsym}
\usepackage{hyperref}
\usepackage{graphicx}
\usepackage{mathrsfs}
\usepackage{changes}
\newtheorem{theorem}{Theorem}[section]
\newtheorem{lemma}[theorem]{Lemma}
\newtheorem{corollary}[theorem]{Corollary}

\newtheorem{proposition}[theorem]{Proposition}

\newcommand{\qed}{\hfill $\Box$ }

\newcommand{\proof}{\noindent{\bf Proof.}\ \ }
\usepackage[algoruled,linesnumbered]{algorithm2e}
\usepackage{algorithmic}

\begin{document}

\title{\Large {\bf  {Isomorphic daisy cubes based on their $\tau$-graphs}}}
\maketitle

\begin{center}
{\large \bf  
Zhongyuan Che$^{a}$,
Niko Tratnik$^{b,c}$,\\ Petra \v Zigert Pleter\v sek$^{d,b}$
}
\end{center}
\bigskip\bigskip

\baselineskip=0.20in

\smallskip

$^a$ {\it Department of Mathematics, Penn State University, Beaver Campus, Monaca, USA} \\

$^b$ {\it University of Maribor, Faculty of Natural Sciences and Mathematics, Slovenia} \\

$^c$ {\it Institute of Mathematics, Physics and Mechanics, Ljubljana, Slovenia} \\

$^d$ {\it University of Maribor, Faculty of Chemistry and Chemical Engineering, Slovenia}\\

\begin{center}
{\tt  zxc10@psu.edu, niko.tratnik@um.si, petra.zigert@um.si}
\end{center}

\begin{abstract}
We prove that if $A$ and $B$ are daisy cubes whose $\tau$-graphs   are forests,
then $A$ and $B$ are isomorphic if and only if their $\tau$-graphs are isomorphic.
The result is applied to show that a daisy cube with at least one edge  is the resonance graph of a plane  bipartite graph $G$
if and only if  its $\tau$-graph is a  forest which is isomorphic to the inner dual of the subgraph of $G$ 
obtained by removing all forbidden edges. As a consequence, some well known properties of Fibonacci cubes and Lucas cubes
are provided as examples with different proofs.
 \vskip 0.1in
{\noindent {\emph{Keywords}}:  daisy cube, $\tau$-graph,  peripherally 2-colorable graph,  resonance graph}
\end{abstract}

\section{Introduction}
 
Daisy cubes and median graphs are two important subfamilies of partial cubes, 
both of them are closely related to the resonance graphs in chemical graph theory.
It is known \cite{ZLS08} that the resonance graph of a plane weakly elementary bipartite graph is a median graph. 
A characterization of plane bipartite graphs whose resonance graphs are daisy cubes was given recently
in \cite{BCTZ25}. More {structural} properties of resonance graphs that are daisy cubes {were provided} in \cite{BCTZ26} 
and \cite{CC25}.

The $\tau$-graphs of median graphs were introduced by Vesel \cite{V05} to characterize resonance graphs
of catacondensed hexagonal graphs: A graph $G$ is the
resonance graph of a catacondensed hexagonal graph if and only if $G$ is a median graph whose $\tau$-graph is
a tree $T$ with vertex degree at most $3$, and the degree-3 vertices of $T$ correspond to the
peripheral $\Theta$-classes of $G$.
Klav\v zar and Kov\v se  \cite{KK07} extended the concept  of the $\tau$-graphs to partial cubes,
and showed that every graph is a $\tau$-graph of some median graph,
the $\tau$-graph of a median graph is $K_n$-free if and only if it does not contain
any convex $K_{1,n}$, and the $\tau$-graph of a graph $G$  is connected if and only if $G$ is a prime graph with respect to
the Cartesian product.

We use $P_n$ and $C_n$ to represent a path and a cycle on $n$ vertices, respectively.
It is known \cite{KK07} that non-isomorphic partial cubes can have the same $\tau$-graph: 
the partial cube $C_6$ (neither a median graph nor a daisy cube), the median graph $K_{1,3}$ (also the Lucas cube $\Lambda_3$),  
and the daisy cube $Q^{-}_3$ (not a median graph) all have the same $\tau$-graph $K_3$,
see Figure \ref{isomorphic_tau}.

\begin{figure}[h!] 
\begin{center}
\includegraphics[scale=.8, trim= 0cm 0.5cm 0cm 0cm]{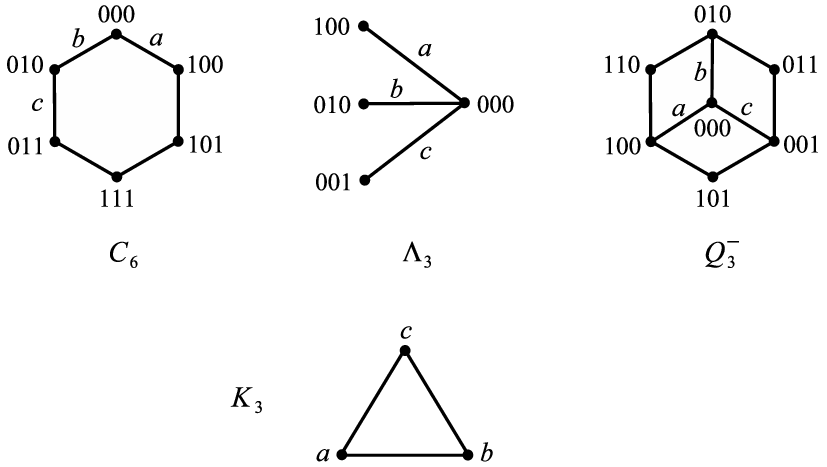}
\end{center}
\caption{\label{isomorphic_tau} Partial cubes $C_6$, $\Lambda_3$, and $Q_3^{-}$ 
with $\Theta$-classes $a,b,c$ have the same $\tau$-graph $K_3$.}
\end{figure}

Besides the example presented in Figure \ref{isomorphic_tau},
we observe that for $n \ge 3$,  a path $P_{n+1}$ (not a daisy cube) 
and a Fibonacci cube $\Gamma_n$ (a daisy cube) are non-isomorphic 
partial cubes with the same $\tau$-graph $P_n$.
We also observe that non-isomorphic daisy cubes can have the same $\tau$-graph, see Figure \ref{isomorphic_tau1}. 

\begin{figure}[h!] 
\begin{center}
\includegraphics[scale=.8, trim= 0cm 0.5cm 0cm 0cm]{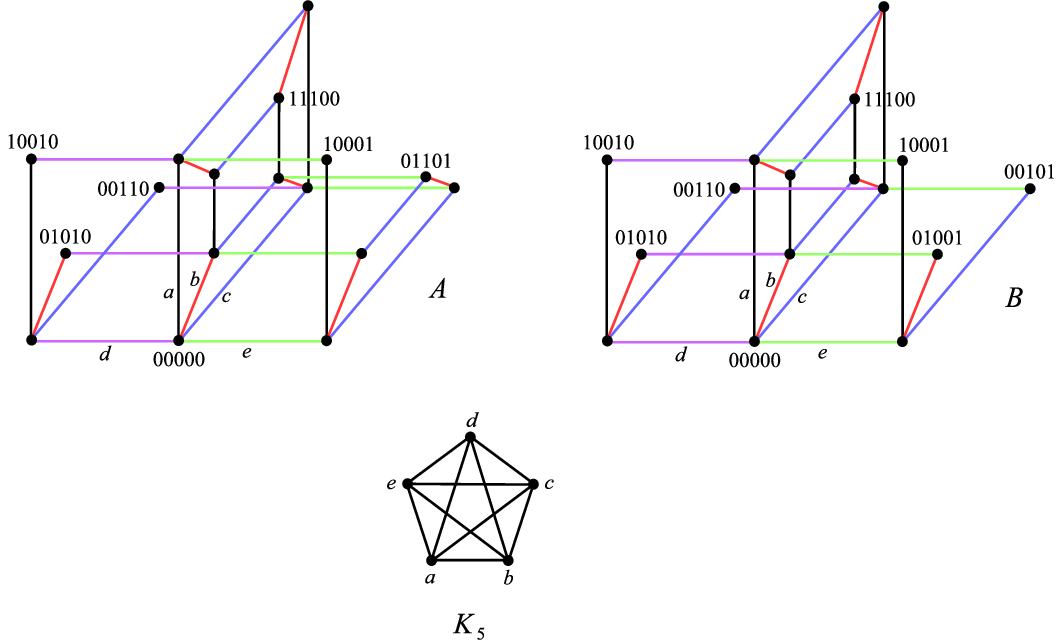}
\end{center}
\caption{\label{isomorphic_tau1} Two non-isomorphic daisy cubes with the same $\tau$-graph $K_5$.}
\end{figure}

In this paper, we focus on studying $\tau$-graphs of daisy cubes. For two daisy cubes $A$ and $B$ whose $\tau$-graphs are forests, 
we prove that $A$ and $B$ are isomorphic if and only if their $\tau$-graphs $A^{\tau}$ and $B^{\tau}$ are isomorphic.
We then apply the result to characterize that a daisy cube {with at least one edge} is the resonance graph of a plane bipartite graph $G$
if and only if the $\tau$-graph of the daisy cube is a forest which is isomorphic to the inner dual of the subgraph of $G$ 
obtained by removing all forbidden edges.
In particular,  a daisy cube is the resonance graph of a plane elementary bipartite graph $G$ {with more than two vertices}
if and only if the $\tau$-graph of the daisy cube is a tree which is isomorphic to the inner dual of $G$.
Related well known properties of $\tau$-graphs of Fibonacci cubes and Lucas cubes are considered as examples with different proofs.

 \section{Preliminaries}\label{S:2}
 
Let $G$ be a simple graph. We use $V(G)$ and $E(G)$ to denote the vertex set and the edge set of $G$, respectively.
A vertex of $G$ is called \textit{non-isolated} if it is adjacent to at least one other vertex of $G$.
A \textit{component} of $G$ is a maximal connected induced subgraph of $G$.
The \textit{complement} of $G$, written as $G^C$, is a graph with the same vertex set as $G$
such that two vertices are adjacent in $G^C$ if and only if they are not adjacent in $G$. 
A subgraph of $G$ induced by a vertex subset $X \subseteq V(G)$ is denoted as $\langle X \rangle$. 
A \textit{clique} $K$ of $G$ is a subgraph such that any two vertices of $K$ are adjacent in $G$.
The \textit{simplex graph}  of $G$, denoted by $\mathcal{K}(G)$, is the graph whose vertices are the cliques 
of $G$ (including the clique on empty set), and two cliques are adjacent if they differ in exactly one vertex.

 An \textit{isomorphism} between two graphs $G_1$ and $G_2$,
 denoted by $f: G_1 \rightarrow G_2$,  is a bijection $f$ between  $V(G_1)$
and $V(G_2)$ such that $u_1v_1 \in E(G_1)$ if and only if $f(u_1)f(v_1) \in E(G_2)$. 
We call that $G_1$ and $G_2$ are \textit{isomorphic} if there exists an isomorphism between them,
and \textit{non-isomorphic} otherwise.

 A \textit{Cartesian product} $\Box_{i=1}^{n} G_i$ of $n$ graphs $G_1, G_2, \ldots G_n$, where $n \ge 2$,
is a graph whose vertex set is $V(G_1) \times V(G_2) \times \cdots \times V(G_n)$ 
and two vertices $(x_1,\ldots, x_n)$ and $(y_1,\ldots, y_n)$ 
are adjacent in  $\Box_{i=1}^{n} G_i$ if there exists an $i \in \{ 1, \ldots, n \}$ such that   $x_iy_i \in E(G_i)$, 
and $x_j=y_j$ for any $j \in \{ 1, \ldots, n \} \setminus \{ i \}$.  
An \textit{$n$-dimensional hypercube} $Q_n$  is a Cartesian product $\Box_{i=1}^{n} K_2$ for $n \ge 2$, 
while $Q_0$ is a one-vertex graph, and $Q_1$ is the one-edge graph.
It is easily seen that $Q_n=\mathcal{K}({K_n})$, where $K_n$ is the complete graph on $n$ vertices.

 Let $G$ be a connected graph. The \textit{distance} $d_G(x,y)$ 
is the length of a shortest path between two vertices $x$ and $y$ of $G$, and
the \textit{interval} $I_G(x,y)$ is the set of all vertices on
shortest paths  between two vertices $x$ and $y$ in $G$. 
A subgraph $H$ of $G$ is an \textit{isometric subgraph} if $d_H(u,v)=d_G(u,v)$ for every two vertices $u$ and $v$ of $H$.
\textit{Partial cubes} are isometric subgraphs of hypercubes. 
A graph $G$ is a \textit{median graph} if $G$
 is a connected graph such that for every triple of vertices $u,v,w$ in $G$, 
 $|I_G(u,v) \cap I_G(u,w) \cap I_G(v,w) | =1$. It is well known  \cite{M80, HIK11} that median graphs are partial cubes.
 
 Let $n$ be a positive integer. Let $(\mathcal{B}^n, \le)$ be a poset on the set of all binary strings of length $n$
with the partial order $u_1u_2 \ldots u_n \leq v_1v_2 \ldots v_n$ if $u_i \leq v_i$ for all $1 \le i \le n$.
A \textit{daisy cube generated by $X  \subseteq \mathcal{B}^n$} is an induced  subgraph of $Q_n$, 
and defined  as $Q_n(X)=\langle \{ u \in \mathcal{B}^{n} \ | \ u \leq x \textrm{ for some } x \in X \} \rangle$ \cite{KM19}. 
The vertex  $0^n$ is the \textit{minimum vertex} of $Q_n(X)$.

\textit{Djokovi\' c-Winkler relation} 
(briefly, \textit{relation $\Theta$}) is a relation on the set of edges of a connected graph $G$ such that
two edges $x_1x_2$ and $y_1y_2$ of  $G$ are  
in \textit{relation $\Theta$} if $d_G(x_1, y_1) + d_G(x_2, y_2) \neq  d_G(x_1, y_2) + d_G(x_2, y_1)$.
The relation $\Theta$ is an equivalence relation on the edge set of any partial cube, see \cite{D73}.
   
 Let  $ab$ be an edge of a connected graph $G$. Then four subsets of the vertex set $V(G)$ can be defined:
\begin{eqnarray*}
W_{ab} &=& \{w \mid w \in V(G), d_G(w,a) < d_G(w,b)\},\\
W_{ba} &=& \{w \mid w \in V(G), d_G(w,b) < d_G(w,a)\},\\
U_{ab} &=& \{u \in W_{ab} \mid u \mbox{ is adjacent to a vertex in } W_{ba}\},\\
U_{ba} &=& \{v \in W_{ba} \mid v \mbox{ is adjacent to a vertex in } W_{ab}\}.
\end{eqnarray*}
Any $\Theta$-class of a partial cube $G$ is a set of edges represented by 
$F_{ab}=\{ e \in E(G) \mid  e \Theta ab\}$ for some edge $ab \in G$ and 
the spanning subgraph $G-F_{ab}$ has exactly two components: $\langle W_{ab} \rangle$ and $\langle W_{ba} \rangle$ \cite{HIK11}. 
Moreover, a $\Theta$-class of a partial cube is called a \textit{peripheral $\Theta$-class} if either $W_{ab}=U_{ab}$ or $W_{ba}=U_{ba}$.

Suppose that $H$ is a graph with $V(H)=V_1 \cup V_2$, where $V_1 \cap V_2 \neq \emptyset$,
$\langle V_1 \rangle$ and $\langle V_2 \rangle $ are isometric subgraphs of $H$,
and there is no edge of $H$ between $ V_1 \setminus V_2$ and $V_2 \setminus V_1$.
The \textit{expansion} of $H$ with respect to $\langle V_1 \cap V_2 \rangle$ is the graph $G$ obtained from $H$ by the following steps:
(i) Replace each $v \in V_1 \cap V_2$ by an edge $v_1v_2$.
(ii) Insert edges between $v_1$ and all neighbors of $v$ in $V_1 \setminus V_2$,
and insert  edges between $v_2$ and all neighbors of $v$ in $V_2 \setminus V_1$.
(iii) Insert the edges $u_1v_1$ and $u_2v_2$ if $u,v \in V_1 \cap V_2$ are adjacent in $H$. See  \cite{HIK11}.
If $H=\langle V_1 \rangle$, then the expansion is called a \textit{peripheral expansion} of $H$ with respect to $H_0=\langle V_2 \rangle$, 
and denoted by $\mathrm{pe}(H,H_0)$. By definition, we can see that if $G=\mathrm{pe}(H,H_0)$, then $H_0$ is an isometric subgraph of  $H$,
and $H$ is an isometric subgraph of $G$ such that
$H_0$ is isomorphic to $G \setminus V(H)$, which is the induced subgraph of $G$ obtained by removing all vertices from $H$.
Moreover, if $H$ is a partial cube, then the edges between $H_0 \subseteq H$ and $G \setminus V(H)$  form the 
$\Theta$-class $\{ e \in E(G) \mid  e \Theta ab\}$ for some edge $ab \in E(G)$
such that $H=\langle W_{ab} \rangle$ and $H_0= \langle U_{ab} \rangle$, 
which is isomorphic to $\langle W_{ba} \rangle= \langle U_{ba} \rangle=G \setminus V(H)$.

If $\langle V_1 \cap V_2 \rangle$ is an isometric subgraph (respectively, a convex subgraph)  of $H$, 
then the expansion is call an isometric expansion
(respectively, a convex expansion).
A graph $G$ is a partial cube if and only if $G$ can be obtained from the one-vertex
graph by a sequence of isometric expansions \cite{C88}. 
A graph $G$ is a median graph if and only if $G$ can be obtained from the one-vertex
graph by a sequence of convex expansions \cite{M78}.

The inverse operation of an expansion in partial cubes $G$ is called a \textit{contraction},
that is, a contraction of a partial cube $G$
is obtained by contracting the edges of a given $\Theta$-class $\mathcal{E}$, {which is denoted by $G/\mathcal{E}$.}
It is known  \cite{T20} that  every $\Theta$-class of a daisy cube is peripheral. 
Therefore, if $G$ is a daisy cube, then $G/\mathcal{E}$ is an induced subgraph of $G$ for any $\Theta$-class $\mathcal{E}$ of $G$.

Let $G$ be a partial cube whose vertex set is a poset $(V(G), \le)$ contained in $(\mathcal{B}^n, \le)$.
Let  $\mathcal{P}(V(G))$ be the power set of $V(G)$.
Then an operator $o: \mathcal{P} (V (G)) \rightarrow \mathcal{P} (V (G))$ can be defined
with $o(X) = \{u \in V(G) \mid  u  \le v \text{ for some } v \in X\}$ for any $X \subseteq V(G)$. 
An induced subgraph $H$ of $G$ is called a \textit{$\le$-subgraph of $G$} if  $V (H) = o(V (H))$. 
Assume that $G$ is a daisy cube and $H$ is a $\le$-subgraph of $G$.
Then the peripheral expansion of $G$ with respect to $H$ is called the \textit{$\leq$-expansion} of $G$ with respect to $H$.
A connected graph is a daisy cube if and only if it can be obtained from the one-vertex graph by a sequence of $\leq$-expansions  \cite{T20}.
It is well known that a contraction of a daisy cube is also a daisy cube \cite{KM19}.

Fibonacci cubes and Lucas cubes are special types of daisy cubes. For each $n \ge 1$,
a {\em Fibonacci cube} $\Gamma_n$  is a graph whose vertex set is 
the set of all binary strings of length $n$ without consecutive 1's, and two vertices 
are adjacent in $\Gamma_n$ if they differ in exactly one position. 
A {\em Lucas cube} $\Lambda_n$  is the induced subgraph of the Fibonacci cube $\Gamma_n$ 
such that the vertex set of $\Lambda_n$ consists of all binary strings 
of length $n$ without consecutive 1's and also without 1 in the first and the last positions. 

The \textit{$\tau$-graph of a partial cube $H$, denoted by $H^{\tau}$},  is a graph
whose vertex set is the set of $\Theta$-classes of $H$, and two distinct  $\Theta$-classes $\mathcal{E}$ and $\mathcal{F}$ are adjacent
in $H^{\tau}$ if $H$ has two edges $uv \in \mathcal{E}$ and $vw \in \mathcal{F}$ such that $uvw$ 
is a convex path on three vertices, that is,  
vertex $v$ is the only common neighbor of two nonadjacent vertices $u$ and $w$ \cite{KK07}. 
Note that the \textit{$\tau$-graph ${R(G)}^{\tau}$} of the resonance graph $R(G)$ of a plane bipartite graph $G$
is equivalent to  the \textit{induced graph $\Theta(R(G))$} defined in 
\cite{C19} as a graph
whose vertex set is the set of $\Theta$-classes of $R(G)$, and two vertices $\mathcal{E}$ 
and $\mathcal{F}$ of $\Theta(R(G))$ are adjacent if  
$R(G)$ has two incident edges $uv \in \mathcal{E}$ and $vw \in \mathcal{F}$ 
such that $uv$ and $vw$  are not contained in a common 4-cycle of $R(G)$.
For clarity and consistence, we will use the notation  \textit{$\tau$-graph ${R(G)}^{\tau}$} introduced in \cite{KK07}.

\section{Isomorphic daisy cubes based on their $\tau$-graphs}

In this section, we prove the main result of the paper. 
Firstly, we show that a daisy cube is a hypercube $Q_n$ if and only if its $\tau$-graph {is an empty graph on $n$ vertices.}

\begin{lemma} \label{lema_cube}
Let $H$ be a daisy cube with $n$ $\Theta$-classes for some positive integer $n$. 
Then $H=Q_n$ if and only if $H^{\tau}=K^C_n$.
\end{lemma}
 \proof
The necessity part is obvious.
Let $H^{\tau}=K^C_n$, which is an empty graph on $n$ vertices. 
Prove the sufficiency by induction on $n$. 
If $n=1$, then any daisy cube with exactly one $\Theta$-class is the one-edge graph which is $Q_1$,
and so $H=Q_1$.
If $n=2$, then any daisy cube with exactly two $\Theta$-classes is either $P_3$ or $Q_2$,
where $P_3^{\tau}=K_2$ and $Q_2^{\tau}=K^C_2$. It follows that $H=Q_2$. 
Suppose that the {sufficiency} is true for any daisy cube with less than $n$ $\Theta$-classes where $n \ge 3$.
Let $H$ be a daisy cube with $n$ $\Theta$-classes such that $H^{\tau}=K^C_n$.
Recall \cite{T20} that a connected graph is a daisy cube if and
only if it can be obtained from the one-vertex graph by a sequence of $\le$-expansions,
and any $\Theta$-class of a daisy cube is peripheral.
Then $H$ can be obtained from a {daisy cube} $H'$ 
{by} a peripheral $\le$-expansion with respect to a {$\le$-subgraph} $H''$ of $H'$,
and a new $\Theta$-class $\mathcal{E}_n$ is generated during the expansion.
It is well known \cite{KM19} that every $\Theta$-class of a daisy cube $H$ has exactly one edge incident to $0^n$.
Then the new $\Theta$-class $\mathcal{E}_n$ has one edge incident to at least one edge from each of other $\Theta$-class of $H$.
Since $H^{\tau}=K^C_n$, it follows that $H''=H'$ and $H=H' \Box K_2$
to avoid any convex path of length 2 with one edge from $\mathcal{E}_n$.
By induction hypothesis, $H'=Q_{n-1}$.
Hence, $H=Q_n$. \qed \\

 \begin{theorem}\label{T:tau-graph-forest-isomorphism}
Let $A$ and $B$ be daisy cubes whose $\tau$-graphs $A^{\tau}$ and $B^{\tau}$ are forests. 
Let $\{\mathcal{E}_1,  \ldots, \mathcal{E}_n\}$ and $\{\mathcal{F}_1,  \ldots, \mathcal{F}_n\}$
be the sets of $\Theta$-classes of $A$ and $B$, respectively. If
$\Upsilon: A^{\tau} \rightarrow B^{\tau}$ is an isomorphism 
such that $\Upsilon(\mathcal{E}_i)=\mathcal{F}_i$
for $i \in \{1, \ldots, n\}$, then there exists an isomorphism $\lambda: A \rightarrow B$ such that 
 $uv \in \mathcal{E}_i$ if and only if $\lambda(u)\lambda(v) \in \mathcal{F}_i$ for $i \in \{1, \ldots, n\}$.
Moreover, the restriction of $\lambda$ on $A/\mathcal{E}_j$ is an isomorphism 
between $A/\mathcal{E}_j$ and $B/\mathcal{F}_j$ for any $j \in \{1, \ldots, n\}$.
\end{theorem}

\proof  By the definition of a $\tau$-graph, we know that the vertex set of the $\tau$-graph 
of a daisy cube is the set of $\Theta$-classes of the daisy cube.
Let $n$ be the number of $\Theta$-classes of each of daisy cubes $A$ and $B$. 
Then each of their $\tau$-graphs $A^{\tau}$ and $B^{\tau}$ has $n$ vertices.
By Lemma \ref{lema_cube},   if $A^{\tau}$ and $B^{\tau}$ are $K^C_n$, then
$A$ and $B$ are $Q_n$, and so the result  holds true.
We call a forest nontrivial if it contains at least one edge. 
Let $A^{\tau}$ and $B^{\tau}$ be nontrivial forests with $n$ vertices.
Then $n \ge 2$  since a daisy cube with exactly one $\Theta$-class is the one-edge graph, 
and its $\tau$-graph is the one-vertex graph which is a trivial forest.

Prove by induction on $n \ge 2$. If $n=2$, a daisy cube with exactly two $\Theta$-classes is either $P_3$ or $Q_2$.
Since $P_3^{\tau}=K_2$, and $Q_2^{\tau}=K^C_2$, 
a daisy cube with {exactly} two $\Theta$-classes and whose $\tau$-graph is a nontrivial forest must be $P_3$,  and
so the result holds true.
We assume that the result holds true for  daisy cubes with less than $n$ $\Theta$-classes
where $n \ge 3$ and whose $\tau$-graphs are nontrivial forests.
 
Let $A$ and $B$ be daisy cubes with exactly $n$ $\Theta$-classes where $n \ge 3$ and  
whose $\tau$-graphs $A^{\tau}$ and $B^{\tau}$ are isomorphic nontrivial forests.
We observe that  $\{\mathcal{E}_1,  \ldots, \mathcal{E}_n\}$ and $\{\mathcal{F}_1,  \ldots, \mathcal{F}_n\}$
are the vertex sets of  $\tau$-graphs $A^{\tau}$ and $B^{\tau}$, respectively.
If $\Upsilon:  A^{\tau} \to  B^{\tau}$  is   an isomorphism between $A^{\tau}$  and $B^{\tau}$
 such that   $\Upsilon(\mathcal{E}_i)=\mathcal{F}_i$ for $i \in \{1, \ldots, n\}$, then
by the assumption that $A^{\tau}$ and $B^{\tau}$ are nontrivial forests,
we can let $\mathcal{E}_n$ and $\mathcal{F}_n$  be {non-isolated} vertices of $A^{\tau}$ and
$B^{\tau}$ respectively such that $\Upsilon(\mathcal{E}_n)=\mathcal{F}_n$.
Let  $\Upsilon'$ be the restriction of $\Upsilon$ on $A^{\tau} \setminus \{\mathcal{E}_n\}$.
 Then  $\Upsilon': A^{\tau} \setminus \{\mathcal{E}_n\} \rightarrow B^{\tau} \setminus \{\mathcal{F}_n\}$ is an isomorphism
 between two forests  $A^{\tau} \setminus \{\mathcal{E}_n\}$ and $B^{\tau} \setminus \{\mathcal{F}_n\}$
such that $\Upsilon'(\mathcal{E}_i)=\mathcal{F}_i$ for $i \in \{1, \ldots, n-1\}$.
 
By Proposition 2.1 in \cite{T20}, every $\Theta$-class of a daisy cube is peripheral. 
Then a contraction of any $\Theta$-class of a daisy cube results an induced subgraph of the daisy cube.
Let $\bar{A}=A/\mathcal{E}_n$ (respectively, $\bar{B}=B/\mathcal{F}_n$) be the induced subgraph of  
the daisy cube $A$ (respectively, the daisy cube $B$) 
obtained by contracting edges of the $\Theta$-class $\mathcal{E}_n$ (respectively, the $\Theta$-class $\mathcal{F}_n$).
Then both $\bar{A}$ and $\bar{B}$ are daisy cubes since a contraction of a daisy cube is also a daisy cube by Proposition 2.2 in \cite{KM19}.
By definitions, a contraction is the inverse operation of an expansion.
Then daisy cube $A$ can be obtained from daisy cube $\bar{A}$ by a peripheral expansion 
with respect to {an isometric} subgraph $\bar{\mathbb{A}}$ of $\bar{A}$. 
A new peripheral $\Theta$-class $\mathcal{E}_n$ is generated during the expansion
which is contained in the subgraph $ \bar{\mathbb{A}}  \Box K_2$ of $A$.
It is clear that $\bar{A}$ is a {nonempty} proper induced subgraph of $A$ by the definition of a peripheral expansion.
We observe that $\bar{\mathbb{A}}$ must be a {nonempty} proper induced subgraph of $\bar{A}$. 
Otherwise,  each edge of $\mathcal{E}_n$ is contained in a 4-cycle of $A$, and so $\mathcal{E}_n$ is an isolated vertex
of $A^{\tau}$. This is a contradiction to the assumption that $\mathcal{E}_n$  is a {non-isolated} vertex of $A^{\tau}$.
Similarly, we can see that daisy cube $B$ can be obtained from a daisy cube $\bar{B}$ by a peripheral expansion 
with respect to {an isometric} subgraph $ \bar{\mathbb{B}}$ of  $\bar{B}$.
A new {peripheral} $\Theta$-class $\mathcal{F}_n$ is generated during the expansion
which is contained in the subgraph $ \bar{\mathbb{B}} \Box K_2$ of $B$.
Moreover, $\bar{B}$ is a nonempty proper induced subgraph of $B$,
and $\bar{\mathbb{B}}$ is a nonempty proper induced subgraph of $\bar{B}$.

For each $i \in \{1, \ldots, n-1\}$, let $\bar{\mathcal{E}}_i=\mathcal{E}_i \cap E(\bar{A})$ 
and $\bar{\mathcal{F}}_i=\mathcal{F}_i \cap E(\bar{B})$. 
Then  $\{\bar{\mathcal{E}}_1, \ldots,  \bar{\mathcal{E}}_{n-1}\}$ and
 $\{\bar{\mathcal{F}}_1, \ldots,  \bar{\mathcal{F}}_{n-1}\}$
are the sets of $\Theta$-classes of daisy cubes $\bar{A}$ and $\bar{B}$, respectively.
Let $\bar{A}^{\tau}$ and $\bar{B}^{\tau}$ be $\tau$-graphs of $\bar{A}$ and $\bar{B}$, respectively.

\smallskip
\smallskip
{\bf Claim 1.} 
There exists an isomorphism $\bar{\Upsilon}$: $\bar{A}^{\tau} \rightarrow \bar{B}^{\tau}$
between two forests $\bar{A}^{\tau}$ and $\bar{B}^{\tau}$  and  such that 
$\bar{\Upsilon}(\bar{\mathcal{E}}_i)=\bar{\mathcal{F}}_i$
for $i \in \{1, \ldots, n-1\}$.  
\smallskip
\smallskip

\textit{Proof of Claim 1.}  The conclusion is obvious if $\bar{A}^{\tau}$ and $\bar{B}^{\tau}$ are $K^C_{n-1}$.
 Let $\bar{A}^{\tau}$ and  $\bar{B}^{\tau}$ be nontrivial forests.
Let $f: {\bar{A}}^{\tau} \rightarrow A^{\tau} \setminus \{\mathcal{E}_n\}$ such that $f(\bar{\mathcal{E}}_i)=\mathcal{E}_i$ for $i \in \{1, \ldots, n-1\}$.
We will show that $f$ is an isomorphism between $\bar{A}^{\tau}$ and $A^{\tau} \setminus \{\mathcal{E}_n\}$. 
Note that this is not necessarily true if $A^{\tau}$ is not a forest, see Figure \ref{figure3}.

\begin{figure}[h!] 
\begin{center}
\includegraphics[scale=.8, trim= 0cm 0.5cm 0cm 0cm]{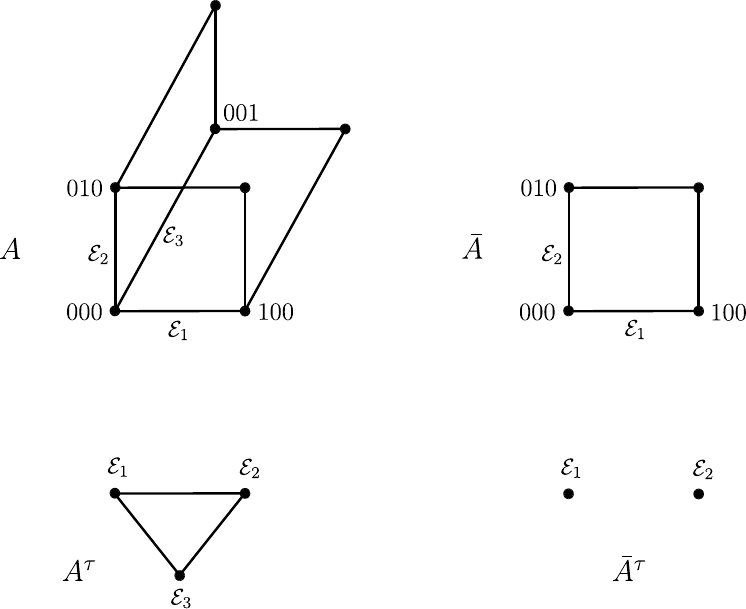}
\end{center}
\caption{\label{figure3} Daisy cube $A$ with $\Theta$-classes $\{ \mathcal{E}_1, \mathcal{E}_2, \mathcal{E}_3 \} $ 
such that $A^{\tau} \setminus \{ \mathcal{E}_3 \} $ is not isomorphic to $\bar{A}^{\tau}$.}
\end{figure}

We first show that if $\bar{\mathcal{E}}_i$ and $ \bar{\mathcal{E}}_j$ are adjacent in $\bar{A}^{\tau}$, 
where distinct $i, j \in \{1, \ldots, n-1\}$, then  $\mathcal{E}_i$ and $\mathcal{E}_j$ are adjacent in $A^{\tau} \setminus \{\mathcal{E}_n\}$. 
This can be seen as follows.
If $\bar{\mathcal{E}}_i$ and $ \bar{\mathcal{E}}_j$ are adjacent in $\bar{A}^{\tau}$,
then $\bar{A}$ contains a convex path $\bar{u}\bar{v}\bar{w}$ of length 2, 
where $\bar{u}\bar{v} \in \bar{\mathcal{E}}_i$ and $\bar{v}\bar{w} \in \bar{\mathcal{E}}_j$
such that $\bar{u}\bar{v}\bar{w}$ is not contained in any 4-cycle of $\bar{A}$.
It remains to show that $\bar{u}\bar{v}\bar{w}$  is also a convex path of $A$.
It is well known \cite{HIK11} that for a partial cube, two incident edges of a 4-cycle are contained in different 
$\Theta$-classes, and two antipodal edges of a 4-cycle are contained in the same $\Theta$-class. 
Then any  4-cycle of $A$ that contains $\bar{u}\bar{v}\bar{w}$ must have all edges from $\mathcal{E}_i \cup \mathcal{E}_j$
since $\bar{u}\bar{v} \in \bar{\mathcal{E}}_i \subseteq \mathcal{E}_i$ 
and $\bar{v}\bar{w} \in  \bar{\mathcal{E}}_j \subseteq  \mathcal{E}_j$.
By the definition of $\bar{A}$ which is the daisy cube obtained the contraction of $\mathcal{E}_n$ of $A$, 
we observe that the spanning subgraph $A - \mathcal{E}_n$ of $A$ has exactly two components
$\bar{A}$ and $A \setminus V(\bar{A})$.
It follows that any 4-cycle of $A$ containing $\bar{u}\bar{v}\bar{w}$ must be  a 4-cycle of $\bar{A}$.
Since $\bar{u}\bar{v}\bar{w}$  is not contained in any 4-cycle of $\bar{A}$, 
it follows that $\bar{u}\bar{v}\bar{w}$ is not contained in any 4-cycle of $A$.
Then $\bar{u}\bar{v}\bar{w}$  is also a convex path of length 2 in $A$, 
and so  $\mathcal{E}_i$ and $ \mathcal{E}_j$ are adjacent in $A^{\tau}$.
It follows that $\mathcal{E}_i$ and $ \mathcal{E}_j$ are adjacent in $A^{\tau} \setminus \{\mathcal{E}_n\}$.

We then show that if $\mathcal{{E}}_i$ and $\mathcal{{E}}_j$ are adjacent in $A^{\tau} \setminus \{\mathcal{E}_n\}$,
where  $i, j \in \{1, \ldots, n-1\}$ and $i \neq j$, then $\bar{\mathcal{E}}_i$ and $\bar{\mathcal{E}}_j$  are adjacent in $ \bar{A}^{\tau}$.
Prove by contradiction. Suppose that $\mathcal{{E}}_i$ and $\mathcal{{E}}_j$   
are adjacent in $A^{\tau} \setminus \{\mathcal{E}_n\}$, 
but $\bar{\mathcal{E}}_i$ and $\bar{\mathcal{E}}_j$  are not adjacent in $ \bar{A}^{\tau}$.
Then any two incident edges of $\bar{A}$, 
where one edge in $\bar{\mathcal{E}}_i$ and the other in $\bar{\mathcal{E}}_j$ 
are contained in a 4-cycle of  $\bar{A}$,
while there exists an edge $uv \in \mathcal{E}_i$ 
and an edge $vw \in \mathcal{E}_j$ such that $uvw$ is a convex path of length $2$ in $A$.
Recall that the spanning subgraph $A - \mathcal{E}_n$ of $A$ has exactly two components
$\bar{A}$ and $A \setminus V(\bar{A})$.
Then the convex path $uvw$ of $A$ is contained in  $A \setminus V(\bar{A})$.
We observe that  $A \setminus V(\bar{A})$ is {isomorphic to} $\bar{\mathbb{A}}$,
and the subgraph $\bar{\mathbb{A}}  \Box K_2$ of $A$ is generated on the vertex set of $\bar{\mathbb{A}} \cup (A \setminus V(\bar{A}))$.
Then there exists a path $\bar{u}\bar{v}\bar{w}$ in $\bar{\mathbb{A}} \subset \bar{A}$ 
such that the induced subgraph $\langle \bar{u}\bar{v}\bar{w} \cup uvw \rangle$ is isomorphic to 
$\bar{u}\bar{v}\bar{w} \Box K_2$ and contained in $\bar{\mathbb{A}}  \Box K_2$. 
It is clear that $u \bar{u}$, $v \bar{v}$, $w \bar{w}$ are edges of $\mathcal{E}_n$. 
Recall that $uv \in \mathcal{E}_i$ and $vw \in \mathcal{E}_j$.
Then $\bar{u}\bar{v} \in \bar{\mathcal{E}}_i$ and $\bar{v}\bar{w} \in \bar{\mathcal{E}}_j$.
By our assumption that $\bar{\mathcal{E}}_i$ and $\bar{\mathcal{E}}_j$  are not adjacent in $ \bar{A}^{\tau}$,
it follows that $\bar{u}\bar{v}\bar{w}$ is contained in a 4-cycle  $\bar{u}\bar{v}\bar{w}\bar{x}$ in $\bar{A}$. 
Suppose that $u \bar{u} \bar{x}$ is  contained in a 4-cycle of $A$. 
Then the 4-cycle must be $u \bar{u} \bar{x}x$  where $\bar{x}x \in \mathcal{E}_n$,
since  two antipodal edges of a 4-cycle are contained in the same $\Theta$-class. 
It follows that $uvwx$ is a 4-cycle  in  $A \setminus V(\bar{A})$ containing the convex path $uvw$ of $A$, which is a contradiction. 
Therefore, $u \bar{u} \bar{x}$ cannot be contained in a 4-cycle of $A$. 
Similarly, $w \bar{w} \bar{x}$ cannot be contained in a 4-cycle of $A$.
Now, $A$ contains two convex paths $u \bar{u} \bar{x}$ and $w \bar{w} \bar{x}$ of length 2.
Recall that two edges $u \bar{u}$ and $w \bar{w}$ of $A$ are contained in $\mathcal{E}_n$. 
Note that $\bar{u} \bar{x} \in \bar{\mathcal{E}}_j \subseteq  \mathcal{E}_j$
 and $\bar{w} \bar{x} \in \bar{\mathcal{E}}_i \subseteq \mathcal{E}_i$.
Then $\mathcal{E}_n$  is adjacent to two distinct $\Theta$-classes of $A$ in $A^{\tau}$: 
$\mathcal{E}_i$ and $\mathcal{E}_j$ where distinct $i, j \in \{1, \ldots, n-1\}$. 
Then  $\mathcal{E}_n$, $\mathcal{E}_i$, and $\mathcal{E}_j$ form a 3-cycle in $A^{\tau}$,  
since  $\mathcal{E}_i$ and $\mathcal{E}_j$ are also adjacent in $A^{\tau}$.
This is  a contradiction to the assumption that  $A^{\tau}$ is a forest.

Therefore, we have shown that if $f: \bar{A}^{\tau} \rightarrow A^{\tau} \setminus \{\mathcal{E}_n\}$ 
such that $f(\bar{\mathcal{E}}_i)=\mathcal{E}_i$ for $i \in \{1, \ldots, n-1\}$, 
then $f$ is an isomorphism  between two forests $\bar{A}^{\tau}$ and $A^{\tau} \setminus \{\mathcal{E}_n\}$.
Let $g: \bar{B}^{\tau} \rightarrow B^{\tau} \setminus \{\mathcal{F}_n\}$ 
such that $g(\bar{\mathcal{F}}_i)=\mathcal{F}_i$ for $i \in \{1, \ldots, n-1\}$.  
Similarly, we can show that $g$ is an isomorphism between two forests $\bar{B}^{\tau}$ and $B^{\tau} \setminus \{\mathcal{F}_n\}$. 

Recall that $\Upsilon': A^{\tau} \setminus \{\mathcal{E}_n\} \rightarrow B^{\tau} \setminus \{\mathcal{F}_n\}$  
is an isomorphism where $\Upsilon'$ is the restriction of $\Upsilon$ on $A^{\tau} \setminus \{\mathcal{E}_n\}$
such that $\Upsilon'(\mathcal{E}_i)=\mathcal{F}_i$ for $i \in \{1, \ldots, n-1\}$.
Let $\bar{\Upsilon}:= g^{-1} \circ \Upsilon' \circ f$.
Then $\bar{\Upsilon}: \bar{A}^{\tau} \rightarrow \bar{B}^{\tau}$ is an isomorphism
such that $\bar{\Upsilon}(\bar{\mathcal{E}}_i)= \bar{\mathcal{F}}_i$ for $i \in \{1, \ldots, n-1\}$.
Therefore, $\bar{A}^{\tau}$ and  $\bar{B}^{\tau}$ are isomorphic forests. 
This ends the proof of Claim 1.

\smallskip
\smallskip

Therefore, the following statements hold true  trivially if $\bar{A}^{\tau}$ and  $\bar{B}^{\tau}$ are trivial forests since
$\bar{A}$ and $\bar{B}$ are $Q_{n-1}$ by Lemma \ref{lema_cube},
or by induction hypothesis if $\bar{A}^{\tau}$ and  $\bar{B}^{\tau}$ are not trivial forests:
There  exists an isomorphism $\bar{\lambda}: \bar{A} \rightarrow \bar{B}$ 
such that $uv \in \bar{\mathcal{E}}_i$ if and only if $\bar{\lambda}(u)\bar{\lambda}(v) \in \bar{\mathcal{F}}_i$
for  $i \in \{1, \ldots, n-1\}$.
Moreover, the restriction of $\bar{\lambda}$ on  $\bar{A}/\bar{\mathcal{E}}_j$   
is an isomorphism between $\bar{A}/\bar{\mathcal{E}}_j$ and $\bar{B}/\bar{\mathcal{F}}_j$  for any $j \in \{1, \ldots, n-1\}$.

  \smallskip
\smallskip

{\bf Claim 2.}  Both $\bar{\mathbb{A}}$ and $\bar{\mathbb{B}}$ are daisy cubes.

 \smallskip
\smallskip

\textit{Proof of Claim 2.} 
We know that $A=\mathrm{pe}(\bar{A}, \bar{\mathbb{A}})$, where $ \bar{\mathbb{A}}$ is a 
nonempty proper induced subgraph of daisy cube $\bar{A}$.
Let $ab$ be an edge of $\Theta$-class $\mathcal{E}_n$ of $A$. Then $\mathcal{E}_n=\{e \in E(A) \mid  e \Theta ab \}$. 
Without loss of generality suppose that $\bar{A}=\langle W_{ab} \rangle$ and $ \bar{\mathbb{A}}=\langle U_{ab} \rangle$.
Then $\langle W_{ba} \rangle=\langle U_{ba} \rangle=A \setminus V(\bar{A})$ which is isomorphic to $ \bar{\mathbb{A}}$.
It follows that $|W_{ab}|> |W_{ba}|$.
By Proposition 3.1 in \cite{T20}, we can see that the minimum vertex of $A$ is contained in $\bar{A}=\langle W_{ab} \rangle$.
By Proposition 2.8   in \cite{T20}, it follows that  $\bar{\mathbb{A}}=\langle U_{ab} \rangle$ is a $\le$-subgraph of daisy cube $A$.
It is clear that $\bar{\mathbb{A}}=\langle U_{ab} \rangle$ is also a $\le$-subgraph of daisy cube $\bar{A}$.
Hence,  $A, \bar{A}, \bar{\mathbb{A}}$ have  the same  minimum vertex which is contained in  $\bar{\mathbb{A}}$.
Similarly,  we have that $B=\mathrm{pe}(\bar{B}, \bar{\mathbb{B}})$ 
where $ \bar{\mathbb{B}}$ is a $\le$-subgraph of daisy cubes $B$ and $\bar{B}$,
and $B, \bar{B}, \bar{\mathbb{B}}$ have the same the minimum vertex which is contained in  $\bar{\mathbb{B}}$.
By Proposition 2.7  in \cite{T20}, both $ \bar{\mathbb{A}}$ and $ \bar{\mathbb{B}}$ are daisy cubes
since any $\le$-subgraph of a daisy cube is isomorphic to a daisy cube.
This ends the proof of Claim 2.

 \smallskip
\smallskip
 
We further let $\mathcal{E}_n$ (respectively, $\mathcal{F}_n$) be a degree-1 vertex of $A^{\tau}$ (respectively, $B^{\tau}$).
Then $\mathcal{E}_n$ (respectively,  $\mathcal{F}_n$)  is adjacent to exactly one vertex $\mathcal{E}_{\alpha(n)}$ 
(respectively,  $\mathcal{F}_{\alpha(n)}$) of $A^{\tau}$ (respectively, $B^{\tau}$).

\begin{figure}[h!] 
\begin{center}
\includegraphics[scale=.9, trim= 0cm 0.5cm 0cm 0cm]{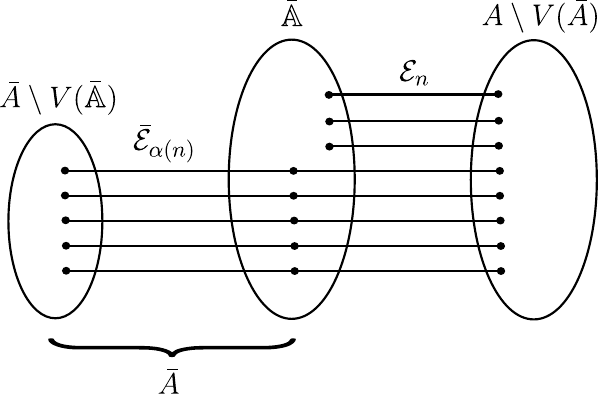}
\end{center}
\caption{\label{structure} Illustration of Claim 3.}
\end{figure}

\smallskip
\smallskip
{{\bf Claim 3.}  
$\bar{\mathcal{E}}_{\alpha(n)}=\mathcal{E}_{\alpha(n)}$ and $\bar{\mathcal{F}}_{\alpha(n)}=\mathcal{F}_{\alpha(n)}$.
The spanning subgraph $\bar{A}- \mathcal{E}_{\alpha(n)}$ of daisy cube $\bar{A}$ has exactly two components} 
$\bar{\mathbb{A}}$ and $\bar{A} \setminus  V(\bar{\mathbb{A}})$.
{The spanning subgraph $\bar{B}- \mathcal{F}_{\alpha(n)}$ of daisy cube $\bar{B}$}
has exactly two components $\bar{\mathbb{B}}$ and $\bar{B} \setminus  V(\bar{\mathbb{B}})$.
\smallskip
\smallskip

\textit{Proof of Claim 3. } 
Note that  $\bar{\mathbb{A}}$ is a {nonempty} proper induced subgraph of daisy cube $\bar{A}$.
Since $\bar{A}$ is connected,
there exists a vertex $\bar{w}$ of $\bar{A} \setminus V(\bar{\mathbb{A}})$ such that $\bar{w}$ is adjacent to a vertex $\bar{u}$ of $\bar{\mathbb{A}}$. 
By the well known facts  that any two edges on a shortest path of a bipartite graph cannot be in the same $\Theta$-class \cite{HIK11},
we can see that two edges incident to the same vertex $\bar{w}$ cannot be in the same $\Theta$-class of $\bar{A}$. 
It follows that if a vertex $\bar{w}$ of $\bar{A} \setminus V(\bar{\mathbb{A}})$ is adjacent to two distinct vertices $\bar{u}, \bar{v}$ of $\bar{\mathbb{A}}$, 
then $\bar{w}\bar{u} \in \bar{\mathcal{E}}_i \subseteq \mathcal{E}_i$ and $\bar{w}\bar{v} \in \bar{\mathcal{E}}_j \subseteq \mathcal{E}_j$
for distinct $i, j \in \{1, 2, \ldots, n-1\}$.
Note that $A=\mathrm{pe} (\bar{A}, \bar{\mathbb{A}})$ such that the set of edges between $\bar{\mathbb{A}}$ and $A \setminus V(\bar{A})$ 
is a matching which forms
the $\Theta$-class $\mathcal{E}_n$ of $A$.
Since $\bar{u}, \bar{v} \in V(\bar{\mathbb{A}})$, 
it follows that there exist two vertices $u,v$ of {$A \setminus V(\bar{A})$}
such that edges $\bar{u}u, \bar{v}v \in \mathcal{E}_n$. 
Since $\bar{w}$ is a vertex of $\bar{A} \setminus V(\bar{\mathbb{A}})$,
$\bar{w}$ cannot be incident to any edge contained in the $\Theta$-class $\mathcal{E}_n$.
Since two antipodal edges of a 4-cycle are contained in the same $\Theta$-class,
it follows that $\bar{w}\bar{u}u$ and $\bar{w}\bar{v}v$ cannot be contained in a 4-cycle.
Hence, $\bar{w}\bar{u}u$ and $\bar{w}\bar{v}v$ are two convex paths of length 2 in $A$, and $\mathcal{E}_n$ is adjacent to 
both $\mathcal{E}_i$ and $\mathcal{E}_j$ in $A^{\tau}$. This is a contradiction to our assumption that $\mathcal{E}_n$ 
is a degree-1 vertex of $A^{\tau}$.

If a vertex $\bar{u}$ of $\bar{\mathbb{A}}$
 is adjacent to two distinct vertices 
$\bar{x}, \bar{y}$ of $\bar{A} \setminus V(\bar{\mathbb{A}})$, 
then $\bar{u}\bar{x} \in \bar{\mathcal{E}}_i \subseteq \mathcal{E}_i$ and $\bar{u}\bar{y} \in \bar{\mathcal{E}}_j \subseteq \mathcal{E}_j$
for distinct $i, j \in \{1, 2, \ldots, n-1\}$ since two edges incident to the same vertex $\bar{u}$ cannot be in the same $\Theta$-class. 
Note that $A=\mathrm{pe} (\bar{A}, \bar{\mathbb{A}})$  such that the set of edges between $\bar{\mathbb{A}}$ and $A \setminus V(\bar{A})$ 
is a matching which forms 
the $\Theta$-class $\mathcal{E}_n$ of $A$, and $\bar{u}  \in V(\bar{\mathbb{A}})$. 
It follows that there exists a vertex $u$ of $A$
such that the edge $u\bar{u}  \in \mathcal{E}_n$. Since $\bar{x}, \bar{y} \in \bar{A} \setminus V(\bar{\mathbb{A}})$,
neither $\bar{x}$ or $\bar{y}$ can be incident to an edge contained in the $\Theta$-class $\mathcal{E}_n$.
It follows that $u\bar{u}\bar{x}$ and $u\bar{u}\bar{y}$ are two convex paths of length 2, and $\mathcal{E}_n$ is adjacent to 
both $\mathcal{E}_i$ and $\mathcal{E}_j$ in $A^{\tau}$. This is a contradiction to our assumption that $\mathcal{E}_n$ 
is a degree-1 vertex of $A^{\tau}$.

As a consequence, the edges between $\bar{A} \setminus V(\bar{\mathbb{A}})$
and $\bar{\mathbb{A}}$ are pairwise vertex disjoint edges.
We observe that each edge between $\bar{A} \setminus V(\bar{\mathbb{A}})$ 
and $\bar{\mathbb{A}}$ can be written as $\bar{x}\bar{w}$ where $\bar{x}$ is a vertex of $\bar{A} \setminus V(\bar{\mathbb{A}})$,
$\bar{w}$ is a vertex of $\bar{\mathbb{A}}$, and $\bar{x}\bar{w}w$ is a convex path of $A$ where   $\bar{w} w \in \mathcal{E}_n$.
Since $\mathcal{E}_n$ is a degree-1 vertex  and adjacent to $\mathcal{E}_{\alpha(n)}$ in $A^{\tau}$,
it follows that the set of edges between $\bar{A} \setminus V(\bar{\mathbb{A}})$
and $\bar{\mathbb{A}}$ must be $\mathcal{E}_{\alpha(n)}$.
Then the spanning subgraph $A-\mathcal{E}_{\alpha(n)}$ of $A$ 
has exactly two components $\bar{A} \setminus V(\bar{\mathbb{A}})$ 
and $\bar{\mathbb{A}} \cup (A \setminus V(\bar{A}))$, where each edge of 
$\mathcal{E}_{\alpha(n)}$ has one end vertex in $\bar{A} \setminus V(\bar{\mathbb{A}})$
and the other end vertex in $\bar{\mathbb{A}}$ (see also \cite{O08}).
It follows that $\mathcal{E}_{\alpha(n)}=\bar{\mathcal{E}}_{\alpha(n)}$ since
the spanning subgraph $\bar{A}-\bar{\mathcal{E}}_{\alpha(n)}$ of $\bar{A}$ has exactly two components  
$\bar{A} \setminus V(\bar{\mathbb{A}})$ and $\bar{\mathbb{A}}$, where each edge of 
$\bar{\mathcal{E}}_{\alpha(n)}$ has one end vertex in $\bar{A} \setminus V(\bar{\mathbb{A}})$
and the other end vertex in $\bar{\mathbb{A}}$.
 Similarly we can show that $\mathcal{F}_{\alpha(n)}=\bar{\mathcal{F}}_{\alpha(n)}$,
 while the spanning subgraph $B-\mathcal{F}_{\alpha(n)}$ of $B$ 
has exactly two components $\bar{B} \setminus V(\bar{\mathbb{B}})$ 
and $\bar{\mathbb{B}} \cup (B \setminus V(\bar{B}))$
 and the spanning subgraph $\bar{B}-\mathcal{F}_{\alpha(n)}$ of $\bar{B}$ has exactly two components
$\bar{B} \setminus  V(\bar{\mathbb{B}})$ and $\bar{\mathbb{B}}$. This ends the proof of Claim 3.

\smallskip
\smallskip

{\bf Claim 4.}  Daisy cube $\bar{\mathbb{A}}$ (respectively, $\bar{\mathbb{B}}$) has exactly $n-2$ $\Theta$-classes 
$\bar{\mathbb{E}}_i=\bar{\mathcal{E}}_i \cap E(\bar{\mathbb{A}})$ 
(respectively, $\bar{\mathbb{F}}_i=\bar{\mathcal{F}}_i \cap E(\bar{\mathbb{B}})$)
 for all $i \in\{1,\ldots, n\} \setminus \{n, \alpha(n)\}$. 
 
(Note that this is not necessarily true if $A^{\tau}$ (respectively, {$B^{\tau}$}) is not a forest or $\mathcal{E}_n$ 
(respectively, $\mathcal{F}_n$) is not a degree-1 vertex of $A^{\tau}$ (respectively, {$B^{\tau}$}).
For example, see Figure \ref{figure3} where daisy cube $\bar{\mathbb{A}}$  is induced on vertices $000$, $100$, and $010$,
daisy cube $\bar{A}$ is the 4-cycle containing daisy cube $\bar{\mathbb{A}}$ and contained in daisy cube $A$.
Both  $\bar{\mathbb{A}}$ and $\bar{A}$ have $n-1$ $\Theta$-classes while $A$ has $n$ $\Theta$-classes, where $n=3$.)
 
\smallskip
\smallskip

\textit{Proof of Claim 4.}  
{Recall that $\bar{A}$ has $n-1$ $\Theta$-classes $\{\bar{\mathcal{E}}_1, \ldots, \bar{\mathcal{E}}_{n-1}\}$,
where $\bar{\mathcal{E}}_i=\mathcal{E}_i \cap E(\bar{A})$ for $i \in \{1, 2 \ldots, n-1\}$. 
By Proposition 2.1 in \cite{T20}, every $\Theta$-class of a daisy cube is peripheral. 
By the proof of Claim 2, we know that $\bar{A}, \bar{\mathbb{A}}, A$ have the same minimum vertex. 
By Proposition 3.1 in \cite{T20} and Claim 3, we can see that $\bar{A}$ can be obtained from $\bar{\mathbb{A}}$ by a peripheral expansion,  
and a new $\Theta$-class $\bar{\mathcal{E}}_{\alpha(n)}=\mathcal{E}_{\alpha(n)}$ is generated during the process.
So, $\bar{\mathbb{A}}=\bar{A}/\bar{\mathcal{E}}_{\alpha(n)}$ where 
$\bar{\mathcal{E}}_{\alpha(n)}=\mathcal{E}_{\alpha(n)}$.
Recall that $\bar{A}=A/\mathcal{E}_n$.
Then $\bar{\mathbb{A}}$ can be obtained from daisy cube $A$ by contracting two peripheral 
$\Theta$-classes $\mathcal{E}_n$ and $\mathcal{E}_{\alpha(n)}$.}
Let $\bar{\mathbb{E}}_i=\bar{\mathcal{E}}_i \cap E(\bar{\mathbb{A}})=\mathcal{E}_i \cap E(\bar{\mathbb{A}})$ 
for each $i \in \{1,\ldots, n\} \setminus \{n, \alpha(n)\}$.
It follows that daisy cube  $\bar{\mathbb{A}}$ has exactly $n-2$ $\Theta$-classes $\bar{\mathbb{E}}_i$,
 where $\bar{ \mathbb{E}}_i \subset \bar{\mathcal{E}}_i \subset \mathcal{E}_i$ for $i \in \{1,\ldots, n\} \setminus \{n, \alpha(n)\}$.
Let $\bar{\mathbb{F}}_i=\bar{\mathcal{F}}_i \cap E(\bar{\mathbb{B}})=\mathcal{F}_i \cap E(\bar{\mathbb{B}})$ 
for each $i \in\{1,\ldots, n\} \setminus \{n, \alpha(n)\}$.
Similarly,  daisy cube $\bar{\mathbb{B}}$  has exactly $n-2$ $\Theta$-classes  $\bar{ \mathbb{F}}_i$, where
$\bar{ \mathbb{F}}_i \subset \bar{\mathcal{F}}_i\subset \mathcal{F}_i$ for $i \in \{1,\ldots, n\} \setminus \{n, \alpha(n)\}$.
This ends the proof of Claim 4.
 
 \smallskip
\smallskip
 
We have shown that  there  exists an isomorphism $\bar{\lambda}: \bar{A} \rightarrow \bar{B}$ 
such that $uv \in \bar{\mathcal{E}}_i$ if and only if $\bar{\lambda}(u)\bar{\lambda}(v) \in \bar{\mathcal{F}}_i$
for  $i \in \{1, \ldots, n-1\}$.
Moreover, the restriction of $\bar{\lambda}$ on $\bar{A}/\bar{\mathcal{E}}_j$  
is an isomorphism between $\bar{A}/\bar{\mathcal{E}}_j$  and  $\bar{B}/\bar{\mathcal{F}}_j$ for any $j \in \{1, \ldots, n-1\}$. 
When $j=\alpha(n)$, we have shown that $\bar{\mathbb{A}}=\bar{A}/\bar{\mathcal{E}}_{\alpha(n)}$ 
where $\bar{\mathcal{E}}_{\alpha(n)}=\mathcal{E}_{\alpha(n)}$,
and $\bar{\mathbb{B}}=\bar{B}/\bar{\mathcal{F}}_{\alpha(n)}$ 
where $\bar{\mathcal{F}}_{\alpha(n)}=\mathcal{F}_{\alpha(n)}$. 
Let $\bar{\bar{\lambda}}$ be the restriction of $\bar{\lambda}$ on $\bar{\mathbb{A}}$.
It follows that  $\bar{\bar{\lambda}}: \bar{\mathbb{A}} \rightarrow \bar{\mathbb{B}}$ is an isomorphism
such that $uv \in \bar{\mathbb{E}}_i= \mathcal{E}_i \cap E(\bar{\mathbb{A}})$ 
if and only if $\bar{\bar{\lambda}}(u)\bar{\bar{\lambda}}(v) \in \bar{\mathbb{F}}_i=\mathcal{F}_i \cap E(\bar{\mathbb{B}})$ 
for  $i \in \{1,\ldots, n\} \setminus \{n, \alpha(n)\}$.

Now, we show that there exists an isomorphism $\lambda: A \rightarrow B$ such that 
 $uv \in \mathcal{E}_i$ if and only if $\lambda(u)\lambda(v) \in \mathcal{F}_i$ for any $i \in \{1, \ldots, n\}$.
 
Recall that $A=\mathrm{pe}(\bar{A}, \bar{\mathbb{A}})$ and $B=\mathrm{pe}(\bar{B},  \bar{\mathbb{B}})$.
By the definition of a peripheral expansion, we can see that the spanning subgraph $A - \mathcal{E}_n$ of $A$ has 
exactly two components
$\bar{A}$ and $A \setminus V(\bar{A})$, and the edge set $\mathcal{E}_n$ is 
a matching between  $\bar{\mathbb{A}}$ and $A \setminus V(\bar{A})$
which induces an isomorphism $h_{\mathbb{A}}: \bar{\mathbb{A}} \rightarrow A \setminus V(\bar{A})$
such that $uv \in \bar{ \mathbb{E}}_i=\mathcal{E}_i \cap E(\bar{\mathbb{A}})$ if and only if  
$h_{\mathbb{A}}(u) h_{\mathbb{A}}(v) \in \mathcal{E}_i \cap E(A \setminus V(\bar{A}))$
for $i \in \{1,\ldots, n\} \setminus \{n, \alpha(n)\}$.
Similarly,  the spanning subgraph $B - \mathcal{F}_n$ of $B$ has 
exactly two components
$\bar{B}$ and $B \setminus V(\bar{B})$,  and the edge set $\mathcal{F}_n$ is 
a matching between  $\bar{\mathbb{B}}$ and $B \setminus V(\bar{B})$
which induces an isomorphism $h_\mathbb{B}: \bar{\mathbb{B}} \rightarrow B \setminus V(\bar{B})$
such that $uv \in \bar{ \mathbb{F}}_i=\mathcal{F}_i \cap E(\bar{\mathbb{B}})$ 
if and only if  $h_{\mathbb{B}}(u) h_{\mathbb{B}}(v) \in \mathcal{F}_i \cap E(B \setminus V(\bar{B}))$
for $i \in \{1,\ldots, n\} \setminus \{n, \alpha(n)\}$.
Let $\bar{\bar{\lambda}}^c = h_{\mathbb{B}} \circ \bar{\bar{\lambda}} \circ  h^{-1}_{\mathbb{A}}$, 
where $\bar{\bar{\lambda}}: \bar{\mathbb{A}} \rightarrow \bar{\mathbb{B}}$ 
is the isomorphism obtained by the restriction of $\bar{\lambda}: \bar{A} \rightarrow \bar{B}$ on $\bar{\mathbb{A}}$
such that $uv \in  \mathcal{E}_i \cap E(\bar{\mathbb{A}})$ 
if and only if $\bar{\bar{\lambda}}(u)\bar{\bar{\lambda}}(v) \in \mathcal{F}_i \cap E(\bar{\mathbb{B}})$ 
for  $i \in \{1,\ldots, n\} \setminus \{n, \alpha(n)\}$.
Then $\bar{\bar{\lambda}}^c: A \setminus V(\bar{A}) \rightarrow B \setminus V(\bar{B})$ is an isomorphism
such that $uv \in \mathcal{E}_i \cap E(A \setminus V(\bar{A}))$ 
if and only if $\bar{\bar{\lambda}}^c(u) \bar{\bar{\lambda}}^c(v) \in  \mathcal{F}_i \cap E(B \setminus V(\bar{B}))$ 
for $i \in \{1,\ldots, n\} \setminus \{n, \alpha(n)\}$. 

Define $\lambda: A \rightarrow B$ such that the restriction of $\lambda$ on $\bar{A}$ is $\bar{\lambda}$, 
which is an isomorphism from $\bar{A}$ to $\bar{B}$,
and the restriction of $\lambda$ on $A \setminus V(\bar{A})$ is $\bar{\bar{\lambda}}^c$, 
which is an isomorphism from $A \setminus V(\bar{A})$ to $B \setminus V(\bar{B})$.
For simplicity, we write $\lambda=\bar{\lambda} \cup \bar{\bar{\lambda}}^c$.
By the properties of $\bar{\lambda}$ and $\bar{\bar{\lambda}}^c$, we can see that  
for $i \in \{1, \ldots, n-1\}$,
$uv \in \mathcal{E}_i =  (\mathcal{E}_i \cap E(\bar{A})) \cup (\mathcal{E}_i \cap E(A \setminus V(\bar{A})))$ 
if and only if $\lambda(u) \lambda(v) \in \mathcal{F}_i =  (\mathcal{F}_i \cap E(\bar{B})) \cup (\mathcal{F}_i \cap E(B \setminus V(\bar{B})))$.
To show that $\lambda$ is an isomorphism between $A$ and $B$, 
it remains to show that $uv \in \mathcal{E}_n$   if and only if $\lambda(u)\lambda(v) \in \mathcal{F}_n$.
Let $uv \in \mathcal{E}_n$.
Then without loss of generality, we can assume that  $u \in V(\bar{\mathbb{A}})$ 
and  $v=h_{\mathbb{A}}(u) \in V(A \setminus V(\bar{A}))$.
Then $\lambda(u)=\bar{\bar{\lambda}}(u)  \in V(\bar{\mathbb{B}})$ and 
$\lambda(v)=\bar{\bar{\lambda}}^c(v)=\bar{\bar{\lambda}}^c(h_{\mathbb{A}}(u)) 
= h_{\mathbb{B}} (\bar{\bar{\lambda}}(u))  \in V(B \setminus V(\bar{B}))$.
So,  $\lambda(u)\lambda(v)=\bar{\bar{\lambda}}(u)\bar{\bar{\lambda}}^c(v) 
=\bar{\bar{\lambda}}(u) h_{\mathbb{B}} (\bar{\bar{\lambda}}(u)) \in \mathcal{F}_n$.
Hence,   $uv \in \mathcal{E}_n$ implies that $\lambda(u)\lambda(v) \in \mathcal{F}_n$.
On the other hand, if $\lambda(u)\lambda(v) \in \mathcal{F}_n$, then
without loss of generality,
we can assume that  $\lambda(u) \in V(\bar{\mathbb{B}})$ 
and $\lambda(v)=h_{\mathbb{B}}(\lambda(u)) \in V(B \setminus V(\bar{B}))$.
Since $\bar{\bar{\lambda}}: \bar{\mathbb{A}} \rightarrow \bar{\mathbb{B}}$ is an isomorphism 
obtained by the restriction of $\lambda$ on $\bar{\mathbb{A}}$,
it follows that if  $\lambda(u)\in V(\bar{\mathbb{B}})$, then
$\lambda(u)=\bar{\bar{\lambda}}(u)$ where $u \in V(\bar{\mathbb{A}})$.
Since $\bar{\bar{\lambda}}^c: A \setminus V(\bar{A}) \rightarrow B \setminus V(\bar{B})$ is an isomorphism 
obtained by the restriction of $\lambda$ on $A \setminus V(\bar{A})$,
it follows that if $\lambda(v)=h_{\mathbb{B}}(\lambda(u)) \in V(B \setminus V(\bar{B}))$, 
then $v \in V(A \setminus V(\bar{A}))$ and $\lambda(v)=\bar{\bar{\lambda}}^c(v)$.
Moreover, $\lambda(v)=h_{\mathbb{B}}(\bar{\bar{\lambda}}(u))=\bar{\bar{\lambda}}^c(h_{\mathbb{A}}(u))$.
Now $\bar{\bar{\lambda}}^c(v)=\bar{\bar{\lambda}}^c(h_{\mathbb{A}}(u))$.
Then $v=h_{\mathbb{A}}(u)  \in V(A \setminus V(\bar{A}))$ since 
$\bar{\bar{\lambda}}^c: A \setminus V(\bar{A}) \rightarrow B \setminus V(\bar{B})$ is an isomorphism.
Hence,  $\lambda(u)\lambda(v) \in \mathcal{F}_n$ implies that  $uv=u h_{\mathbb{A}}(u) \in \mathcal{E}_n$.
Therefore,  $\lambda: A \rightarrow B$ is an isomorphism such that $\lambda(\mathcal{E}_i)=\mathcal{F}_i$ for $i \in \{1, \ldots, n\}$. 

Finally, we show that the restriction of $\lambda$ on $A/\mathcal{E}_j$ is an isomorphism
between $A/\mathcal{E}_j$ and  $B/\mathcal{F}_j$  for any $j \in \{1, \ldots, n\}$.

By the definition of  $\lambda$, we know that the restriction of $\lambda$ on $\bar{A}$ 
is the isomorphism $\bar{\lambda}: \bar{A} \rightarrow \bar{B}$ 
where $\bar{A}=A/\mathcal{E}_n$ and $\bar{B}=B/\mathcal{F}_n$.
The conclusion holds true when $j=n$.

Note that $A/\mathcal{E}_{\alpha(n)}$ is the induced subgraph of $A$ 
generated on the vertex set of $\bar{\mathbb{A}} \cup (A \setminus V(\bar{A}))$,
and $B/\mathcal{F}_{\alpha(n)}$ is the induced subgraph of $B$ 
generated on the vertex set of $\bar{\mathbb{B}} \cup (B \setminus V(\bar{B}))$. 
Recall that the restriction of   $\bar{\lambda}$ on $\bar{\mathbb{A}}$
is an isomorphism $\bar{\bar{\lambda}}: \bar{\mathbb{A}} \rightarrow \bar{\mathbb{B}}$ 
such that $uv \in \mathcal{E}_i \cap E(\bar{\mathbb{A}})$ 
if and only if $\bar{\bar{\lambda}}(u) \bar{\bar{\lambda}}(v) \in  \mathcal{F}_i \cap  E(\bar{\mathbb{B}})$ 
for $i \in \{1, \ldots, n\} \setminus \{\alpha(n), n\}$. 
Moreover, $\bar{\bar{\lambda}}^c: A \setminus V(\bar{A}) \rightarrow B \setminus V(\bar{B})$ is an isomorphism
such that $uv \in \mathcal{E}_i \cap E(A \setminus V(\bar{A}))$ 
if and only if $\bar{\bar{\lambda}}^c(u) \bar{\bar{\lambda}}^c(v) \in  \mathcal{F}_i \cap E(B \setminus V(\bar{B}))$ 
for $i \in \{1, \ldots, n\} \setminus \{\alpha(n), n\}$. 
We observe that if $uv \in \mathcal{E}_n$, then $u \in V(\bar{\mathbb{A}})$ 
and  $v=h_{\mathbb{A}}(u) \in V(A \setminus V(\bar{A}))$; and
if $\lambda(u)\lambda(v) \in \mathcal{F}_n$, then $\lambda(u) =\bar{\bar{\lambda}}(u) \in V(\bar{\mathbb{B}})$ 
and  $\lambda(v)=\bar{\bar{\lambda}}^c(v) \in V(B \setminus V(\bar{B}))$.
We also have shown that $uv \in \mathcal{E}_n$  if and only if $\lambda(u)\lambda(v)=\bar{\bar{\lambda}}(u)\bar{\bar{\lambda}}^c(v) \in \mathcal{F}_n$.
It follows that the restriction of $\lambda=\bar{\bar{\lambda}} \cup \bar{\bar{\lambda}}^c$ on $A/\mathcal{E}_{\alpha(n)}$
is an isomorphism between $A/\mathcal{E}_{\alpha(n)}$
and $B/\mathcal{F}_{\alpha(n)}$. So,  the conclusion holds true when $j=\alpha(n)$.
 
It remains to show that the restriction of $\lambda=\bar{\lambda} \cup \bar{\bar{\lambda}}^c$ on $A/\mathcal{E}_j$ is an isomorphism
between $A/\mathcal{E}_j$ and  $B/\mathcal{F}_j$ for any $j \in \{1, \ldots, n\} \setminus \{\alpha(n), n\}$.
Recall that $\lambda: A \rightarrow B$ is an isomorphism 
such that $uv \in \mathcal{E}_i$ if and only if $\lambda(u)\lambda(v) \in \mathcal{F}_i$ for any $i \in \{1, \ldots, n\}$.
Moreover,
$\bar{\lambda}: \bar{A} \rightarrow \bar{B}$ is an isomorphism obtained by the restriction of $\lambda$ on $\bar{A}$,
$\bar{\bar{\lambda}}: \bar{\mathbb{A}} \rightarrow \bar{\mathbb{B}}$ 
is an isomorphism obtained by the restriction of   $\bar{\lambda}$ on $\bar{\mathbb{A}}$,
and $\bar{\bar{\lambda}}^c: A \setminus V(\bar{A}) \rightarrow B \setminus V(\bar{B})$
is an isomorphism obtained by the restriction of $\lambda$ on $A \setminus V(\bar{A})$.
Recall that $\bar{A}$ is a daisy cube with $n-1$ $\Theta$-classes
and  $\bar{\mathbb{A}}$ is a daisy cube with $n-2$ $\Theta$-classes.
Then for any $j \in \{1, \ldots, n\} \setminus \{\alpha(n), n\}$,
the restriction of the isomorphism $\bar{\lambda}: \bar{A} \rightarrow \bar{B}$
on  $\bar{A}/(\mathcal{E}_j \cap E(\bar{A}))$ is  an isomorphism
between $\bar{A}/(\mathcal{E}_j \cap E(\bar{A}))$ and $\bar{B}/(\mathcal{F}_j \cap E(\bar{B}))$;
and the restriction of $\bar{\bar{\lambda}}$ on $\bar{\mathbb{A}}/(\mathcal{E}_j \cap E(\bar{\mathbb{A}}))$ 
is an isomorphism between $\bar{\mathbb{A}}/(\mathcal{E}_j \cap E(\bar{\mathbb{A}}))$ and 
$\bar{\mathbb{B}}/(\mathcal{F}_j \cap E(\bar{\mathbb{B}}))$.
By the definition of the isomorphism $\bar{\bar{\lambda}}^c: A \setminus V(\bar{A}) \rightarrow B \setminus V(\bar{B})$,
it follows that the restriction of $\bar{\bar{\lambda}}^c$ on  $(A \setminus V(\bar{A}))/(\mathcal{E}_j \cap E(A \setminus V(\bar{A})))$ 
is an isomorphism  between $(A \setminus V(\bar{A}))/(\mathcal{E}_j \cap E(A \setminus V(\bar{A})))$ 
and  $(B \setminus V(\bar{B}))/(\mathcal{F}_j \cap E(B \setminus V(\bar{B})))$ for any $j \in \{1, \ldots, n\} \setminus \{\alpha(n), n\}$.
Note that $A/\mathcal{E}_j$ is an induced subgraph of $A$ on the union of the vertex sets of
$\bar{A}/(\mathcal{E}_j \cap E(\bar{A})) $ and $(A \setminus V(\bar{A}))/(\mathcal{E}_j \cap E(A \setminus V(\bar{A})))$,
and $B/\mathcal{F}_j$ is an induced subgraph of $B$ on the union of the vertex sets of
$\bar{B}/(\mathcal{F}_j \cap E(\bar{B})) $ and $(B \setminus V(\bar{B}))/(\mathcal{F}_j \cap E(B \setminus V(\bar{B})))$.
It follows that the restriction of $\lambda=\bar{\lambda} \cup \bar{\bar{\lambda}}^c$ on $A/\mathcal{E}_j$ is an isomorphism
between $A/\mathcal{E}_j$ and  $B/\mathcal{F}_j$ for any $j \in \{1, \ldots, n\} \setminus \{\alpha(n), n\}$.
Therefore, the conclusion holds true for all $j \in \{1, \ldots, n\}$.
\qed \\

By the above theorem we immediately get the following corollary.
\begin{corollary} \label{main_cor}
Let $A$ and $B$ be daisy cubes whose $\tau$-graphs $A^{\tau}$ and $B^{\tau}$ are forests. 
Then $A$ and $B$ are isomorphic {if and only if} $A^{\tau}$ and $B^{\tau}$ are isomorphic.
\end{corollary}

\section{Daisy cubes that are resonance graphs}

In this final section, we characterize daisy cubes that are resonance graphs of plane bipartite graphs. We start with some basic definitions.

A \textit{perfect matching} (or, a $1$-factor) of a graph $G$ is a subset $M \subseteq E(G)$  of edges 
such that every vertex of $G$ is incident with exactly one edge of $M$.  
An edge is said to be \textit{allowed} if it belongs to at least one perfect matching of $G$; otherwise it is called \textit{forbidden}.  
A graph with a perfect matching is called \textit{elementary} if the subgraph induced by all allowed edges is connected.  
It is known  \cite{LP86} that a bipartite graph $G$ is elementary if and only if it is connected and all of its edges are allowed. 
For any graph that admits a perfect matching, each component of the subgraph obtained by removing all forbidden edges is elementary.  

Let $G$ be a plane graph. A face $s$ of $G$ is a region enclosed by a set of edges of $G$, 
which forms the \textit{periphery of $s$}.
A face $s$ of $G$  is called \textit{finite} if the periphery of $s$ encloses a finite region,  
and \textit{infinite} otherwise. The \textit{inner dual} of $G$  is a graph whose vertices are finite faces of $G$ 
such that two vertices are adjacent  if the two corresponding finite faces of $G$ have an edge in common.

The notion of a plane weakly elementary bipartite graph was originally introduced in \cite{ZZ00} for connected graphs.  
For practical reasons, the definition is typically extended to graphs that are not connected (see, for example, \cite{ZZY04}).  
A plane bipartite graph (not necessarily connected) with a perfect matching is called \textit{weakly elementary} 
if the removal of all forbidden edges does not create any new finite faces.  
By definition, every plane elementary bipartite graph is weakly elementary.

The {\em resonance graph} (also called \textit{$Z$-transformation graph}) $R(G)$ of a plane bipartite graph $G$ is the graph 
whose vertices are the  perfect matchings of $G$, and two perfect matchings $M_1,M_2$ are adjacent 
if their symmetric difference $M_1 \oplus M_2$ forms the edge set of exactly one cycle 
that is the  periphery of a finite face $s$ of $G$ \cite{ZZY04}.

It is well known \cite{ZZ00} that if $G$ is a plane weakly elementary bipartite graph with elementary components $G_1, G_2, \ldots, G_t$,
then its resonance graph $R(G)$ is the Cartesian product $\Box_{i=1}^{t}R(G_i)$.  For more details see Section 2.3 in \cite{BCTZ25}.

Peripherally 2-colorable graphs were introduced in  \cite{BCTZ25} 
to characterize resonance graphs that are daisy cubes. 
Let $G$ be a plane elementary bipartite graph different from $K_2$. 
Then $G$ is called \textit{peripherally 2-colorable} if  every vertex of $G$ has degree 2 or 3, 
vertices with degree 3 (if exist) are all exterior vertices of $G$,
and $G$ can be properly colored black and white so that two vertices with the same color are nonadjacent,
and vertices with degree 3 (if exist) are alternatively black and white along the clockwise orientation of the periphery of $G$. 

In order to prove the main result of this section, we firstly need the following two lemmas.

\begin{lemma} \label{L:Tree-Forest}
Let $G$ be a plane weakly elementary bipartite graph whose elementary components different from $K_2$
are peripherally 2-colorable graphs $G_1, G_2, \ldots, G_t$ for some positive integer $t$. 
Then $R(G)^{\tau}=R(G_1)^{\tau} \cup R(G_2)^{\tau} \cup  \cdots  \cup R(G_t)^{\tau}$ 
is a forest, where $R(G_i)^{\tau}$ is a tree  and isomorphic to the inner dual   of $G_i$ for $1 \le i \le t$.
In particular, if $G$ is a peripherally 2-colorable graph, then $R(G)^{\tau}$ is a tree and isomorphic to the inner dual of $G$.
\end{lemma}
\proof {It is well known \cite{ZZ00} that}
$R(G)=R(G_1) \Box R(G_2) \Box \cdots \Box R(G_t)$ where $G_1, G_2, \ldots, G_t$ 
are elementary components of $G$ different from $K_2$, 
since $R(K_2)$ is the one vertex graph and contributes a trivial factor to the Cartesian product.
 By Lemma 3.2 in \cite{KK07}, the $\tau$-graph of a Cartesian product is a disjoint union of $\tau$-graphs of its factors,
 then  $R(G)^{\tau}=R(G_1)^{\tau} \cup R(G_2)^{\tau} \cup  \cdots  \cup R(G_t)^{\tau}$.
 
For $1 \le i \le t$,  
by the proof of Theorem 3.5 in  \cite{BCTZ25},
we know that each peripherally 2-colorable graph $G_i$ is either a 2-connected outerplane bipartite graph, or
can be transformed into a 2-connected outerplane bipartite graph $G'_i$ 
such that the resulting 2-connected outerplane bipartite graph $G'_i$ is also peripherally 2-colorable,
the inner dual of $G'_i$ is isomorphic to the inner dual of $G_i$, and $R(G'_i)$ is isomorphic to $R(G_i)$. 
It is clear that the inner dual of any 2-connected outerplane bipartite graph is a tree.
So, the inner dual of any peripherally 2-colorable graph is a tree.
By Theorem 3.4 in \cite{C19},  it follows that for $1 \le i \le t$,  
each $R(G_i)^{\tau}$ is a tree and isomorphic to the inner dual  of $G_i$.
Therefore, $R(G)^{\tau}=R(G_1)^{\tau} \cup R(G_2)^{\tau} \cup  \cdots  \cup R(G_t)^{\tau}$ is a forest isomorphic to
the inner dual of the subgraph of $G$ obtained by removing all forbidden edges.
In particular, if $G$ is a peripherally 2-colorable graph, then $R(G)^{\tau}$ is a tree and isomorphic to the inner dual of $G$.
\qed\\

\begin{lemma}\label{L:DCube-ResonanceG-TauForest}    
Let $H$ be a daisy cube with at least one edge.
If $H$  is isomorphic to the resonance graph of a plane  bipartite graph $G$, then 
its $\tau$-graph $H^{\tau}$ is a  forest
whose components are trees isomorphic to the inner duals of elementary components of $G$ different from $K_2$.
In particular,  if $H$ is isomorphic to the resonance graph of a plane elementary  bipartite graph $G$,
then  $H^{\tau}$ is a  tree isomorphic to the inner dual of $G$.   
\end{lemma}
\proof  
Let $H$ be a daisy cube with at least one edge. 
If $H$ is isomorphic to the resonance graph $R(G)$ of a plane  bipartite graph $G$,
then  $R(G)$ is a daisy cube with at least one edge.
Recall that  \cite{BCTZ25}  $R(G)$ is a daisy cube 
if and only if $G$ is weakly elementary such that any elementary component of $G$ different from $K_2$ is peripherally 2-colorable.
Moreover, $G$ has at least one elementary component different from $K_2$ since $R(G)$ has at least one edge.
Let $G_1, G_2, \ldots, G_t$ be all elementary components of $G$ different from $K_2$ for some positive integer $t$.
Then each $G_i$ is peripherally 2-colorable for  $1 \le i \le t$.
By Lemma \ref{L:Tree-Forest}, $R(G)^{\tau}=\cup_{i=1}^{t} R(G_i)^{\tau}$  is a forest, 
where $R(G_i)^{\tau}$ is a tree  and isomorphic to the inner dual   of $G_i$ for $1 \le i \le t$.
 Since $H$ is isomorphic to $R(G)$, it follows that $H^{\tau}$ is isomorphic to $R(G)^{\tau}$, 
and so $H^{\tau}$ is a forest with the described properties.

 In particular, if $H$ is isomorphic to the resonance graph of a plane elementary  bipartite graph $G$ different from $K_2$,
then its $\tau$-graph $H^{\tau}$ is a  tree isomorphic to the inner dual of $G$. 
 \qed\\

\begin{theorem}\label{T:DaisyCubeForestTree}    
A daisy cube  $H$ {with at least one edge} is isomorphic to the resonance graph of a plane  bipartite graph 
if and only if its $\tau$-graph  $H^{\tau}$ is a  forest. 
In particular, a daisy cube $H$ is isomorphic to the resonance graph of a plane elementary  bipartite graph
different from $K_2$ if and only if  its $\tau$-graph $H^{\tau}$ is a  tree.   
\end{theorem}
\proof The necessity part follows by Lemma \ref{L:DCube-ResonanceG-TauForest}. 
For the sufficiency part, suppose that $H^{\tau}$ is a  forest. 
Let $(H^{\tau})_1, \ldots, (H^{\tau})_t$ be the connected components of $H^{\tau}$. 
For $1 \le i \le t$,  each $(H^{\tau})_i$ is a tree, 
and it is  easy to construct a peripherally 2-colorable graph $G_i$ whose inner dual is isomorphic to $(H^{\tau})_i$.  
Let $G$ be a disjoint union of graphs $G_1, \ldots, G_t$. Then $G$ is a plane weakly elementary bipartite graph.
By Lemma \ref{L:Tree-Forest},  $R(G)^{\tau}=R(G_1)^{\tau} \cup R(G_2)^{\tau} \cup  \cdots  \cup R(G_t)^{\tau}$,
where $R(G_i)^{\tau}$ is a tree isomorphic to the inner dual of $G_i$ for $1 \le i \le t$. 
Since the inner dual of $G_i$ is isomorphic to $(H^{\tau})_i$ for $1 \le i \le t$, 
it follows that $R(G)^{\tau}$ and $H^{\tau}$ are isomorphic forests. 
Therefore, by Corollary \ref{main_cor}, daisy cubes $R(G)$ and $H$ are isomorphic.
\qed\\

Note that the $\tau$-graph of a hypercube $Q_n$ ($n \geq 2$) is $K^C_n$, which is the  graph on $n$ vertices without edges. 
By Theorem \ref{T:DaisyCubeForestTree}, the hypercube $Q_n$ ($n \geq 2$) cannot be the resonance graph of any plane elementary bipartite graph.
The same conclusion can be obtained from Theorem 3.6 in \cite{ZYY14} 
that the distributive lattice $(\mathcal{M}(G), \le)$ on the set of perfect matchings of a plane elementary bipartite graph $G$ is irreducible,
 or Corollary 3.3 in \cite{C18} that the resonance graph of a plane elementary bipartite graph
cannot be the Cartesian product of nontrivial median graphs.
However, any hypercube $Q_n$ ($n \ge 1$) is the resonance graph of a plane weakly elementary  
bipartite graph with $n$ elementary components that are even cycles,
since the resonance graph of any even cycle is $K_2$.

Recall that Fibonacci cubes $\Gamma_n$ and Lucas cubes $\Lambda_n$ are special types of daisy cubes  \cite{KM19}.
Fibonacci cubes  are the resonance graphs of fibonaccenes, i.e., zigzag hexagonal chains \cite{KZ05}.

\begin{lemma}\label{L:Fibonacci-Lucas}
(i) $\Gamma_n^{\tau}=P_{n}$, where $P_n$  is a path on $n$ vertices.

(ii) $\Lambda_n^{\tau}=C_n$ for $n \ge 3$, where  $C_n$ is a cycle on $n$ vertices.

(iii) $\Lambda_n$ cannot be the resonance graph of any plane bipartite graph for $n \ge 3$.
\end{lemma}
\proof {From Theorem 2.3 in \cite{KK07} we know that for any graph $G$, 
the $\tau$-graph of of the simplex graph of the complement of $G$  is isomorphic to $G$, i.e. $G=\mathcal{K}(G^C)^{\tau}$.   
It is known that Fibonacci cube $\Gamma_n=\mathcal{K}(P^C_n)$ \cite{FMZ99, EKM23} and 
Lucas cube $\Lambda_n=\mathcal{K}(C^C_n)$  ($n \ge 3$) \cite{MPZ01}.
Hence, $\Gamma_n^{\tau}=P_{n}$, and $\Lambda_n^{\tau}=C_n$ ($n \ge 3$).
By Theorem \ref{T:DaisyCubeForestTree}, 
it follows that Lucas cube $\Lambda_n$ cannot be the resonance graph of any plane bipartite graph for $n \ge 3$.} \qed\\

\noindent {\bf Remark.} Lemma \ref{L:Fibonacci-Lucas} (i) is Theorem 8 in \cite{V05} given by Vesel  with a different proof.
Lemma \ref{L:Fibonacci-Lucas} (iii) is a strong version of Lemma 11 in \cite{ZOY09} given by Zhang et al. that Lucas cubes
$\Lambda_n (n \ge 3)$ cannot be the resonance graph of any plane elementary bipartite graphs.
\smallskip
\smallskip

We conclude this section with a different proof for the fact that  $\Lambda_n^{\tau}=C_n$ ($n \ge 3$).

\begin{proposition} \label{prop_lucas}
Let $\Lambda_n$ be a Lucas cube where  $n \geq 3$. Then its $\tau$-graph $\Lambda_n^{\tau}$ is a cycle $C_n$.
\end{proposition}
\proof
For each integer $i \in \{ 1, \ldots, n \}$,  let $X_i$ be the $\Theta$-class {of the edges in $E(\Lambda_n)$ 
such that the binary codes of two end vertices of each edge in $X_i$ differ exactly} in the $i$-th position.
We first prove that  $X_1$ is adjacent to $X_2$ and $X_n$ in $\Lambda_n^{\tau}$. 
Obviously,  vertices  with binary codes $0^n$, $010^{n-2}$,  $10^{n-1}$ 
form a convex path on three vertices since $110^{n-2}$ does not exist in a Lucas cube. Therefore, $X_1$ and $X_2$ are adjacent. 
Similarly, we can show that $X_1$ and $X_n$ are adjacent. 
Further, we show that $X_1$ is not adjacent to $X_j$ for each $j \in \{ 3, \ldots, n-1 \}$. Suppose that $X_1$ and $X_j$ are adjacent. 
Then there exist a convex path on vertices 
$1u_2\ldots u_{j-1}0 u_{j+1}\ldots u_n$, $0u_2\ldots u_{j-1}0 u_{j+1}\ldots u_n$, $0u_2\ldots u_{j-1}1 u_{j+1}\ldots u_n$ in $\Lambda_n$, 
where the first edge belongs to $X_1$ and the other edge to $X_j$.  
However, $1u_2\ldots u_{j-1}1 u_{j+1}\ldots u_n$ is also a vertex in  $\Lambda_n$, so we obtain a $4$-cycle, which is a contradiction.
It follows that  $X_1$ cannot be adjacent to $X_j$ in $\Lambda_n^{\tau}$, where $j \in \{ 3, \ldots, n-1 \}$. 
Therefore,  $X_1$ is adjacent only to $X_2$ and $X_n$ in $\Lambda_n^{\tau}$.  
 Similar argument can be applied for any vertex of $\Lambda_n^{\tau}$, and so $\Lambda_n^{\tau}$ is a cycle on $n$ vertices. 
\qed \\

\vskip 0.2in

\noindent{\bf Acknowledgement:} 
Niko Tratnik and Petra \v Zigert Pleter\v sek acknowledge the financial support from the Slovenian Research and Innovation Agency: research programme P1-0297 
and research projects J1-70016 (Niko Tratnik, Petra \v Zigert Pleter\v sek), 
N1-0285 (Niko Tratnik), and J7-50226 (Petra \v Zigert Pleter\v sek).  


\end{document}